\documentclass[12pt]{amsart}
\usepackage{amssymb}

\theoremstyle{plain}
\newtheorem{theorem}{Theorem}[section]
\newtheorem{corollary}[theorem]{Corollary}

\newtheorem{proposition}[theorem]{Proposition}

\theoremstyle{definition}

\newtheorem{remark}[theorem]{Remark}

\theoremstyle{remark}

\newcommand{\abs}[1]{\lvert#1\rvert}
\newcommand{\norm}[1]{\lVert#1\rVert}

\renewcommand{\mid}{\::\:}

\newcommand{\term}[1]{{\textit{\textbf{#1}}}}

\DeclareMathOperator{\Span}{span}

\begin{document}
\baselineskip 18pt

\title{Every operator has almost-invariant subspaces}

\author[A.I.~Popov]{Alexey I. Popov$^1$}
\address[A.I. Popov]{Department of Pure Mathematics,
University of Waterloo,200 University Avenue West,
Waterloo, Ontario, N2L 3G1, Canada}
\email{a4popov@uwaterloo.ca}
\author[A.~Tcaciuc]{Adi Tcaciuc}
\address[A. Tcaciuc]{Mathematics and Statistics Department,
   Grant MacEwan University, Edmonton, Alberta, Canada T5J
   P2P, Canada}
\email{atcaciuc@ualberta.ca}

\thanks{${}^1$ Research supported in part by NSERC (Canada)}


\date{\today.}

\begin{abstract}
We show that any bounded operator $T$ on a separable, reflexive, infinite-dimensional Banach space $X$ admits a rank one perturbation which has an invariant subspace of infinite dimension and codimension. In the non-reflexive spaces, we show that the same is true for operators which have non-eigenvalues in the boundary of their spectrum. In the Hilbert space, our methods produce perturbations that are also small in norm, improving on an old result of Brown and Pearcy.
\end{abstract}

\maketitle

\section{Introduction}\label{intro}

The Invariant Subspace Problem is concerned with the existence of non-trivial, closed, invariant subspaces for a bounded operator acting on an separable, infinite-dimensional Banach space $X$. The existence of such subspaces for compact operators was proved by von Neumann (unpublished work) for Hilbert spaces and by   Aronszajn and Smith \cite{Aronszajn:54} for Banach spaces. Lomonosov \cite{Lomonosov:73} considerably increased the class of operators with an invariant subspace by showing that every non-scalar operator in the commutant of a compact operator acting on a complex Banach space has a hyperinvariant subspace. On the other hand, the first example of an operator on a Banach space without invariant subspaces was found by Enflo (see~\cite{Enf76} and~\cite{Enf87}). Later, Read~\cite{R85} constructed such an example on $\ell_1$, and also gave examples of quasinilpotent operators~\cite{R97} and strictly singular operators~\cite{R99} without invariant subspaces. Another very important example is the recent construction of Argyros and Haydon~\cite{AH} of an infinite dimensional Banach space such that every bounded operator on that space is the sum of a compact operator and a multiple of the identity; in particular, every bounded operator on this space has an invariant subspace. For an overview of the Invariant Subspace Problem we refer the reader to the monograph by Radjavi and Rosenthal~\cite{RR03} or to the more recent book of Chalendar and Partington~\cite{ChP}.

In this paper we examine a closely related problem:

 \emph{Given a bounded linear operator $T$ acting on a complex Banach space $X$, can we perturb it by a finite rank operator $F$ such that $T-F$ has an invariant subspace of infinite dimension and codimension in $X$?}

This problem, or rather an equivalent formulation of it, was introduced in a paper by Androulakis, Popov, Tcaciuc and Troitsky~\cite{APTT}. We briefly restate the basic definitions from that paper. Unless otherwise specified, for the rest of the paper it is assumed that all Banach spaces are complex and all subspaces are norm closed. A subspace $Y$ of a Banach space $X$ is called an \emph{\textbf{almost-invariant subspace}} for $T$ if there exists a finite dimensional subspace $M$ such that $TY\subseteq Y+M$. The minimal dimension of such a subspace $M$ is referred to as the \textbf{\emph{defect}} of $Y$ for $T$.  It is immediate that every finite-dimensional or finite-codimensional $Y\subseteq X$  is almost invariant under $T$. Thus, when we ask whether every operator $T$ on a Banach space $X$ has an almost-invariant subspace, the question is non-trivial only if we restrict our search to  subspaces that are of infinite dimension and codimension in $X$; such subspaces will be called $\emph{\textbf{half-spaces}}$.

Common almost-invariant subspaces for algebras of operators have been studied in~\cite{P2010} and~\cite{MPR}. In particular, it was shown in~\cite{MPR} that if a norm-closed algebra of operators on a Hilbert space has a common almost-invariant half-space then it has an invariant half-space.

The operators of central importance for the work~\cite{APTT} were the Donoghue shifts. Recall that a weighted shift $D\in\mathcal B(\mathcal H)$ is called a \term{Donoghue shift} if its weights $(w_n)$ are all non-zero and $|w_n|\downarrow 0$ (see, e.g., \cite[Section~4.4]{RR03} for more information about Donoghue shifts). The importance of Donoghue shifts for the paper~\cite{APTT} stems from the fact that all the invariant subspaces for such operators are finite dimensional. In particular, these operators do not admit invariant half-spaces.

The following theorem which, in particular, applies to Donoghue shifts, was proved in~\cite{APTT}. Recall that a sequence $\{x_n\}_n$ in a Banach space is called \term{minimal} if, for every~$k\in\mathbb N$, $x_k$ does not belong to the closed linear span of the set $\{x_n\mid n\neq k\}$ (see also \cite[Section~1.f]{Lindenstrauss:77}).
\begin{theorem}\cite[Theorem 3.2]{APTT}, \cite[Remark 1.3]{MPR}\label{oldpaper}
  Let $X$ be a Banach space and $T \in {\mathcal B}(X)$ satisfying the
  following conditions:
  \begin{enumerate}
  \item\label{rho} The unbounded component of the resolvent set $\rho(T)$ contains
      $\{z\in\mathbb C\mid 0<\abs{z}<\varepsilon\}$ for some $\varepsilon>0$.
  \item There is a vector $e\in X$ whose orbit $\{T^{n}e\}_{n=0}^{\infty}$ is a minimal sequence.
  \end{enumerate}
  Then $T$ has an almost-invariant half-space with defect at most one.
\end{theorem}
Initially the theorem had an extra hypothesis condition, that $T$ has no eigenvalues, which was shown to be redundant in~\cite[Remark~1.3]{MPR} by Marcoux, Popov, and Radjavi. The authors also obtained almost-invariant half-spaces for larger classes of operators. In particular, they showed that every quasinilpotent, triangularizable, injective operator has an almost-invariant half-space, and the same holds for bitriangualar operators and  polynomially compact operators on reflexive Banach spaces.

In the present paper we show that every bounded operator $T$ acting on a reflexive Banach space admits an almost-invariant half-space with defect one. For the general Banach spaces, we show that the same holds if the boundary of the spectrum of the operator contains an element which is not an eigenvalue. This shows, in particular, that if $T$ is an operator without invariant subspaces then there exists a rank-one operator $F$ such that $(T-F)$ has an invariant subspace of infinite dimension and infinite codimension. In the Hilbert space, under the same spectral assumption, we show that the operator $F$ can be chosen as small in norm as desired. If the spectrum of $T$ consists of eigenvalues only, then, for every $\varepsilon>0$, we show that there exists a finite-rank operator $F\in\mathcal B(\mathcal H)$ of rank independent of $\varepsilon$ such that $\norm{F}<\varepsilon$ and $(T-F)$ has an invariant subspace of infinite dimension and infinite codimension.

The following simple facts which were proved in~\cite{APTT} will be used in the present paper without further reference.
\begin{proposition}\label{no-reference}
Let $X$ be a Banach space and $T\in\mathcal B(X)$. Then the following is true.
\begin{enumerate}
\item A subspace $Y$ of $X$ is $T$-almost-invariant if and only if $Y$ is  $(T-F)$-invariant for some finite-rank operator $F\in\mathcal B(X)$.
\item If $T$ admits an almost-invariant half-space then so does $T^*$.
\end{enumerate}
\end{proposition}

These facts are immediate in the Hilbert space setting. Indeed, if $P$ is a projection onto the space $Y$ then the operator $F$ in Proposition~\ref{no-reference}(i) can be defined by $F=(I-P)TP$. The statement of Proposition~\ref{no-reference}(ii) follows from the easy fact that if $Y$ is invariant under $(T-F)$, then $Y^\perp$ is invariant under $(T^*-F^*)$.

\section{Almost-invariant half-spaces}\label{main}

 For a Banach space $X$,  $\mathcal{B}(X)$ and $S(X)$  are, respectively, the algebra of all bounded operators on $X$ and the unit sphere of $X$. For $T\in{\mathcal B}(X)$, we write $\sigma(T)$, $\sigma_a(T)$, $\rho(T)$ and $\partial\sigma(T)$ for the spectrum
of~$T$, approximate point spectrum of $T$,  the resolvent set of~$T$ and the topological boundary of the spectrum, respectively. The closed span of a set  $\{x_n\}_n$ of vectors in $X$ is denoted by $[x_n]$.

 Recall that a sequence $(x_n)_{n=1}^{\infty}$ in $X$ is called a \textbf{basic sequence} if any $x\in[x_n]$ can be written uniquely as $x=\sum_{n=1}^{\infty} a_n x_n$, where the convergence is in norm (see \cite[section 1.a]{Lindenstrauss:77} for background on Schader bases and basic sequences). As $[x_{2n}]\cap[x_{2n+1}]=\{0\}$ it is immediate that $[x_{2n}]$ is of both infinite dimension and infinite codimension in $[x_n]$, thus a half-space, and since every Banach space contains a basic sequence, it follows that every infinite dimensional Banach space contains a half-space. We are going to use the following criterion of Kadets and Pe{\l}czy\'{n}ski for a subset of Banach space to contain a basic sequence (see, e.g., \cite[Theorem 1.5.6]{AK06})

 \begin{theorem}[Kadets, Pe{\l}czy\'{n}ski] \label{criterion}
 Let $S$ be a bounded subset of a Banach space $X$ such that $0$ does not belong to the norm closure of $S$. Then the following
 are equivalent:
 \begin{enumerate}
 \item $S$ fails to contain a basic sequence,
 \item The weak closure of $S$ is weakly compact and fails to contain $0$.
 \end{enumerate}
 \end{theorem}

For a nonzero vector $e\in X$ and
$\lambda\in \rho(T)$, we define, following~\cite{APTT}, a vector $h(\lambda,e)$ in $X$ by
\begin{displaymath}
   h(\lambda, e):=\bigl(\lambda I-T\bigr)^{-1}(e).
\end{displaymath}

Also, observe that $\bigl(\lambda I-T\bigr)h(\lambda, e)=e$ for
every $\lambda\in \rho(T)$, so that
\begin{equation}
  \label{Th}
  Th(\lambda,e)=\lambda h(\lambda,e)-e.
\end{equation}
If $A\subseteq\rho(T)$ and we put $Y=\overline{\Span}\bigl\{h(\lambda,e)\colon\lambda\in A\bigr\}$, it follows immediately from the previous identity that $TY\subseteq Y + \Span\{e\}$, thus $Y$ is an almost-invariant subspace of $T$ with defect one. However $Y$ is nor necessarily a half-space. We are going to construct such a $Y$ in the case when $\partial\sigma(T)$ does not consist entirely of eigenvalues, which combined with the results from \cite{APTT} and \cite{MPR} will give a complete answer to the question in the reflexive case.

\begin{remark} \label{shift}
Notice that if $Y$ is an almost-invariant half-space for $T$ then it is also almost-invariant for $T-\lambda I$ for any $\lambda\in\mathbb{C}$, and the defect stays the same.
\end{remark}

We show next that any bounded operator acting on a reflexive Banach space has an almost-invariant half-space with defect one. The essential ingredient is the following theorem.

\begin{theorem}\label{maintheorem}
Let $X$ be a Banach space and let  $T \in
{\mathcal B}(X)$ such that there exists $\mu\in\partial\sigma(T)$ that is not an eigenvalue. Then $T$ admits an almost-invariant half-space with defect one.
\end{theorem}

\begin{proof}
By eventually replacing $T$ with $T-\mu I$ and considering Remark \ref{shift}, we can assume without loss of generality  that $\mu=0$. The idea of the proof is to construct a sequence of vectors $h(\lambda_n,e)$ as above, that after normalization is a basic sequence. To this end pick a sequence $\{\lambda_n\}$ in $\rho(T)$ such that $\lambda_n\longrightarrow 0\in\partial\sigma(T)$. It follows that $\norm{(\lambda_n I-T)^{-1}}\longrightarrow\infty$, and the uniform boundness principle guarantees we can find $e\in S(X)$ such that $\norm{h(\lambda_n,e)}=\norm{(\lambda_n I-T)^{-1}e}\longrightarrow\infty$. In order to simplify notations, denote $h(\lambda_n,e)$ by $h_n$, $\frac{h_n}{\norm{h_n}}$ by $x_n$, and let $S:=(x_n)\subseteq S(X)$. Note that

\begin{equation} \label{conv}
  Tx_n=\frac{1}{\norm{h_n}}Th_n=\frac{\lambda_n h_n}{\norm{h_n}}-\frac{1}{\norm{h_n}}e=
  \lambda_n x_n-\frac{1}{\norm{h_n}}e.
\end{equation}


If $\overline{S}^{weak}$ is not weakly compact, then  it follows from Theorem \ref{criterion} that $S$ contains a basic sequence. Thus, by eventually passing to a subsequence, we can assume that $(x_n)$ is basic. Then $Y:=[x_{2n}]$ is a half-space and $TY\subseteq Y+ [e]$, and we are done.

If $\overline{S}^{weak}$ is weakly compact, then it is weakly sequentially compact by Eberlein-\v{S}mulian theorem, and, by eventually passing to a subsequence, we can assume that $x_n\stackrel{w}{\longrightarrow} z$. Hence $Tx_n\stackrel{w}{\longrightarrow} Tz$. On the other hand, since $\lambda_n\longrightarrow 0$, and $\norm{h_n}\longrightarrow\infty$,  from (\ref{conv}) we have that
\[
Tx_n=\lambda_n x_n-\frac{1}{\norm{h_n}}e\longrightarrow 0.
\]
Therefore $Tz=0$ and since $0$ is not an eigenvalue, we must have $z=0$, and hence $x_n\stackrel{w}{\longrightarrow} 0$. By applying Theorem \ref{criterion} we obtain that $S$ contains a basic sequence and we continue as in the case when $\overline{S}^{weak}$ is not weakly compact.

\end{proof}

\begin{remark}
Note that an operator $T\in\mathcal{B}(X)$ which has no invariant subspaces cannot have any eigenvalues. It follows from the previous theorem that such an operator has an almost-invariant half-space. In particular, all known counterexamples to the invariant subspace problem (e.g. the operators constructed by Enflo or Read) are not counterexamples to the almost-invariant half-space problem.
\end{remark}

\begin{corollary}
Let $T$ be a bounded operator on a reflexive Banach space such that one of $T$ or $T^{*}$ has a boundary point of the spectrum which is not an eigenvalue. Then $T$ has an almost-invariant half-space.
\end{corollary}

\begin{proof}
Since $X$ is reflexive, $T^{**}=T$ and the conclusion follows immediately from the previous theorem and the fact that if $T$ has an almost invariant half-space, so does~$T^{*}$.
\end{proof}

We are now ready to prove the general case.

\begin{theorem} \label{reflexive}
Let $T$ be a bounded operator on a infinite-dimensional, reflexive Banach space. Then $X$ admits an almost invariant half-space with defect one.
\end{theorem}

\begin{proof}
In view of the previous Corollary, we may assume that any point in the $\partial\sigma(T)=\partial\sigma(T^{*})$ is an eigenvalue for both $T$ and $T^{*}$.

 Case 1: $\partial\sigma(T)$ has infinite cardinality.  In this situation, consider $A_1$ and $A_2$, disjoint, countably infinite,  subsets of $\partial\sigma(T)$, say $A_1=\{\lambda_n\}_n$ and $A_2=\{\mu_n\}_n$. Pick eigenvectors $\{x_n\}_n$ in $X$ and $\{f_n\}_n$ in $X^{*}$ such that $Tx_n=\lambda_n x_n$ and $T^{*}f_n=\mu_n f_n$. Clearly both $\{x_n\}_n$ and $\{f_n\}_n$ are linearly independent, thus, in particular, the subspace $Y:=[x_n]$ is infinite dimensional.  It is easy to see that $Y$ is T-invariant and in order to show that it is a half-space it remains to show that $Y$ is infinite codimensional as well. For any $n$ and $k$ in $\mathbb{N}$ we have:
\[
  \lambda_k f_n(x_k)=f_n(\lambda_k x_k)=f_n(Tx_k)=(T^{*}f_n)(x_k)=\mu_nf_n(x_k).
\]

Since for any $n$, $k$ in $\mathbb{N}$ we have $\lambda_k\neq\mu_n$, it follows that $f_n(x_k)=0$. Therefore, for any $n\in\mathbb{N}$, $f_n(Y)=0$, so $Y$ is annihilated by the linearly independent functionals $\{f_n\}_n$. This proves that $Y$ is infinite codimensional.

Case 2: $\partial\sigma(T)$ is finite. In this case $\sigma(T)$ must also be finite, so $\partial\sigma(T)=\sigma(T)$.  Note that in this situation $\sigma(T)$ satisfies the condition $(i)$ in Theorem \ref{oldpaper}. Also, for any $T$-invariant subspace $Y$ of $X$, we have that $\partial(\sigma(T_{|Y}))\subseteq\sigma_a(T_{|Y})\subseteq\sigma_a(T)=\sigma(T)$, hence $\sigma(T_{|Y})$ is also finite and $\sigma(T_{|Y})\subseteq\sigma(T)$.

  For any $n\in\mathbb{N}$, denote by $Y_{n}=\overline{T^{n}X}$, with $Y_0:=X$. We have that each  $Y_n$ is invariant under $T$, $Y_{n+1}=\overline{TY_n}$ and $X\supseteq Y_1\supseteq Y_2\supseteq\dots.$ Also note that for any $j,n\in\mathbb{N}$ and any $y\in Y_j$ we have that $T^{n}(y)\in Y_{j+n}$. We can assume that each $Y_j$ is infinite dimensional. Indeed, otherwise, if $j$ is the smallest index for which $Y_j$ is finite dimensional, then any half-space of $Y_{j-1}$ containing $Y_j$ is an invariant half-space for $T$.  We have that for any $n$, $\sigma(T_{|Y_{n+1}})\subseteq\sigma(T_{|Y_{n}})$ and since $\sigma(T)$ is finite it follows that there exists $k\in\mathbb{N}$ such that $\sigma(T_{|Y_{n}})=\sigma(T_{|Y_{k}})\neq\emptyset$ for any $n\geq k$. Since an almost-invariant half-space for $T_{|Y_{k}}$ is an almost-invariant half-space for $T$, we can assume without loss of generality that $k=0$ and, as we did before, we can also assume that $0\in\sigma(T)$.

   First, we claim that either we can find a vector $z$ such that the orbit $\{T^{n}z\}$ is minimal, or the restriction of $T$ to some $Y_j$ has dense range.  The proof is inspired by a very similar result from \cite{MPR} (see  Theorem 2.5). We can assume $Y_1$ is of finite codimension in $X$, otherwise we can find an invariant half-space for $T$. Hence $Y_1$ is complemented in $X$ and we can write $X=Y_1\oplus Z$, where $Z$ is finite dimensional.

If $Z=\{0\}$ then $T$ has dense range. Otherwise, let $\{z_1,z_2,\dots z_k\}$ be a basis for $Z$ and assume the orbit $\{T^{n}z_j\}_n$ is not minimal for any $1\leq j\leq k$. For any $1\leq j\leq k$ denote by $p_j$ the smallest index such that $T^{p_j}z_j\in[T^n z_j]_{n\neq p_j}$. It is not hard to prove (see e.g. Lemma 2.2 in \cite{MPR}) that for this choice of $p_j$ we actually have that $T^{p_j}z_j\in[T^n z_j]_{n > p_j}$, thus $T^{p_j}z_j\in Y_{p_j+1}$, for any $1\leq j\leq k$. If we let $p_0:=\max\{p_1,  p_2, \dots p_k\}$, it follows that $T^{p_0}z_j=T^{p_{0}-p_{j}}(T^{p_j}z_j)\in Y_{p_0+1}$ for any $1\leq j\leq k$. Therefore, since $\{z_1,z_2,\dots z_k\}$ is a basis for $Z$, we have that $T^{p_0}z\in Y_{p_0+1}$ for any $z\in Z$. We also have that $T^{p_0}y\in Y_{p_0+1}$ for any $y\in Y_1$, and since $X=Y_1\oplus Z$ it follows that $T_{p_0}x\in Y_{p_0+1}$ for any $x\in X$. This means that $\overline{T^{p_0}X}\subseteq Y_{p_0+1}$, so $Y_{p_0}\subseteq Y_{p_0+1}$. On the other hand, $Y_{p_0+1}\subseteq Y_{p_0}$,
therefore $Y_{p_0+1}=Y_{p_0}$ and the last equality means that $T_{|Y_{p_0}}$ has dense range, and this proves the claim.

If there exists a vector $z$ such that the orbit $\{T^{n}z\}$ is minimal, we can apply Theorem~\ref{oldpaper} to conclude that $T$ has an almost-invariant half-space with defect at most one. Otherwise, let $j\in\mathbb{N}$ such that $S:=T_{|Y_j}:Y_j\to Y_j$ has dense range. Therefore $S^{*}$ is injective, hence $0\in\sigma(S)=\sigma(S^{*})=\partial\sigma(S^{*})$ is not an eigenvalue for $S^{*}$. It follows from Theorem \ref{maintheorem} that $S^{*}$ has an almost-invariant half-space with defect at most one and, since $X$ is reflexive, so does $S$, hence also $T$. This concludes the proof of the theorem.

\end{proof}

For bounded operators on a separable Hilbert space $\mathcal{H}$ we have the following immediate corollary.

\begin{corollary}
For any $T\in\mathcal{B}(\mathcal{H})$ there exists an infinite dimensional subspace $Y$ with infinite dimensional orthogonal complement such that $(I-P)TP$ has rank one, where $P$ is the orthogonal projection onto $Y$. Equivalently, relative to the decomposition $\mathcal{H}=Y\oplus Y^{\perp}$, $T$ has the form

$$T=\left[
  \begin{array}{cc}
    * & * \\
    F & * \\
  \end{array}
\right]$$ where $F$ has rank one.
\end{corollary}

\section{Operators on Hilbert spaces}

Let $X$ be a Banach space and $T\in L(X)$. The results in the previous section are concerned with the choice of a rank-one operator $F\in L(X)$ such that $T-F$ has an invariant subspace of infinite dimension and infinite codimension. In the present section, we will show that in the case of a  Hilbert space, our methods also provide operator $F$ which is small in norm.

We start by recalling a result proved by Brown and Pearcy in~\cite{BP71}.

\begin{theorem}[Brown, Pearcy]
Let $T\in\mathcal{B}(\mathcal{H})$ and $\varepsilon>0$. Then there exists a scalar $\lambda$ and a decomposition of $\mathcal{H}=Y\oplus Y^{\perp}$ into infinite dimensional subspaces such that the corresponding matrix representation of $T$ has the form
$T=\left[
  \begin{array}{cc}
    \lambda I+K & * \\
    F & * \\
  \end{array}
\right],$
where $K$ and $F$ are compact and have norms at most $\varepsilon$.
\end{theorem}


Our first result in this section is valid in all Banach spaces.

\begin{proposition}\label{smallnorm}
Let $X$ be a Banach space and let  $T \in
{\mathcal B}(X)$ such that there exists $\lambda\in\partial\sigma(T)$ that is not an eigenvalue.  Then for any $\varepsilon>0$, $T$ has an almost-invariant half-space $Y_\varepsilon$ such that $(T-\lambda I)_{|Y_{\varepsilon}}:Y_{\varepsilon}\to X$ is compact and $\norm{(T-\lambda I)_{|Y_{\varepsilon}}}<\varepsilon$.
\end{proposition}

\begin{proof}
As before, we can assume without loss of generality that $\lambda=0$. With the notations from the proof of Theorem \ref{maintheorem}, we obtain a basic sequence $(x_n)$ in $X$ such that $[x_{n}]$ is an almost-invariant half-space for $T$. Moreover, from (\ref{conv}) in Theorem \ref{maintheorem}, it follows that $Tx_{n}\longrightarrow 0$. Fix $\varepsilon>0$ and denote by $C$ the basis constant of $(x_{n})$. By eventually passing to a subsequnce of $(x_n)$ we can assume that
\[
\sum_{1}^{\infty}||Tx_n||<\frac{\varepsilon}{2C}.
\]

If we denote by $Y_{\varepsilon}:=[x_n]$,  easy calculations (see e.g. \cite{AA}, Lemma 4.59) show that
$\norm{T_{|Y_{\varepsilon}}}<\varepsilon$ and also that $T_{|Y_{\varepsilon}}$ is the norm limit of finite rank operators $TP_{n}$, where for each $n$, $P_n:Y_\varepsilon\to Y_\varepsilon$ is the  projection on  the fist $n$ coordinates of the basis $(x_n)$. Therefore $T_{|Y_{\varepsilon}}$ is compact and the theorem is proved.
\end{proof}


If the underlying space is a Hilbert space, we can get a more specific information about the structure of the operator. 

\begin{theorem}\label{smallness-hilb}
Let $T\in\mathcal{B}(\mathcal{H})$ be such that there exists $\lambda\in\partial\sigma(T)$ which is not an eigenvalue. Then for any $\varepsilon>0$, there exists a decomposition of $\mathcal{H}=Y\oplus Y^{\perp}$ into infinite dimensional subspaces such that the corresponding matrix representation of $T$ has the form
$T=\left[
  \begin{array}{cc}
    \lambda I+K & * \\
    F & * \\
  \end{array}
\right],$
where $K$ is compact, $F$ has rank one, and both have norms at most $\varepsilon$.

Moreover, applying a similarity, we can obtain that $K$ is  diagonal with respect to an orthonormal basis.
\end{theorem}

\begin{proof}
There is no loss of generality in assuming that $\lambda=0$. Let $(x_n)$ be the sequence as in the proof of Theorem~\ref{maintheorem}. Passing to a subsequence, we may assume that $(x_n)$ converges weakly to a vector $x\in\mathcal H$. If $x$ were not equal to zero, we would have by the formula~\eqref{conv} in the proof of Theorem~\ref{maintheorem} that $Tx=0$, which is impossible by the assumptions. Hence, $(x_n)$ is a normalized weakly null sequence. By the Bessaga-Pe\l czy\'nski selection principle, $(x_n)$ has a subsequence equivalent to a block basis of the unit vector basis of $\ell_2$. Therefore, passing to a subsequence, we may assume that the sequence $(x_n)$ obtained in the proof of Theorem~\ref{maintheorem} is equivalent to the orthonormal basis of~$\mathcal H$.

Let $\varepsilon>0$. Repeating the argument of Proposition~\ref{smallnorm}, we obtain, by passing to a subsequence of $(x_n)$, a half-space $Y=[x_n]$ such that, relative to the decomposition $\mathcal H=Y\oplus Y^\perp$, the operator $T$ can be written in the matrix form $T=\left[
  \begin{array}{cc}
    K & * \\
    F & * \\
  \end{array}
\right]$, where $K$ is compact, $F$ is rank one, and $T|_Y:Y\to\mathcal H$ has the norm less than~$\varepsilon$. The last condition clearly implies that $\norm{K}<\varepsilon$ and $\norm{F}<\varepsilon$. Finally, observe that it follows from formula~\eqref{conv} and the fact that $(x_n)$ is equivalent to an orthonormal basis, that a similarity that maps $(x_n)$ to an orthonormal basis and makes $e$ orthogonal to~$Y$ produces the diagonal version of our representation.
\end{proof}

We remark that if an operator $T$ on a Hilbert space does not satisfy the conditions of Theorem~\ref{smallness-hilb}, then it has eigenvalues and, in particular, has invariant subspaces (which, of course, may have finite dimension or finite codimension). Regarding the subspaces that are half-spaces, we have the following version of Theorem~\ref{smallness-hilb} for such operators.

\begin{proposition}\label{non-trivial-ker}
Let $T\in\mathcal{B}(\mathcal{H})$ be such that both  $\partial\sigma(T)$ and $\partial\sigma(T^*)$ consist of eigenvalues. Then for any $\varepsilon>0$, there exists a finite rank operator $F$ such that $\norm{F}<\varepsilon$ and $T-F$ has an invariant subspace of infinite dimension and codimension. The rank of $F$ can be chosen independent of~$\varepsilon$.
\end{proposition}

\begin{proof}
Notice that if the operator $T$ satisfies the conclusion of the proposition, so does $T^*$. Therefore, we may switch from $T$ to $T^{*}$ as needed.

As observed in the proof of Theorem~\ref{reflexive}, if $\sigma(T)$ is infinite and both $\partial\sigma(T)$ and $\partial\sigma(T^*)$ consist of eigenvalues, then $T$ has an invariant subspace of infinite dimension and codimension. So, there is no loss of generality in assuming that $\sigma(T)$ is finite. By the standard spectral theory, $\mathcal H$ splits into finitely many $T$-invariant subspaces which complement each other such that the spectrum of the restriction of $T$ to each of the subspaces is a singleton. Observe that at least one of the invariant subspaces is infinite dimensional, and it is enough to obtain the conclusion of the proposition for the restriction of $T$ to one such infinite dimensional invariant subspace. Therefore, by subtracting a multiple of the identity if necessary, we may assume without loss of generality that $T$ is quasinilpotent. Notice that in this case $\sigma_{ess}(T)=\{0\}$, where  $\sigma_{ess}(T)$ is the essential spectrum of $T$.

Denote the null space of $T$ by $N$ and the closure of the range of $T$ by $R$. It is clear that, without loss of generality, $n=\dim N<\infty$ and $m=\mathrm{codim} R<\infty$. Further, replacing $T$ by $T^*$, if necessary, we may assume that $n\leq m$. Fix $\varepsilon>0$, and we claim that there exists an operator $G\in\mathcal B(\mathcal H)$ of rank $n$ such that $\norm G<\varepsilon/2$ and $(T+G)$ is injective. Indeed, let $f_1,\dots, f_n$ be a basis for the space $N$ and $g_1,\dots,g_m$ a basis for $R^\perp$. Define $G:N\to R^\perp$ by $G(f_i)=g_i$, where $i\in\{1,\dots,n\}$. Clearly, $G$ is a well-defined (bounded) operator. This can be extended to a bounded operator in $\mathcal B(\mathcal H)$ by letting $G|_{N^\perp}=0$. A straightforward verification shows that $T+\alpha G$ is injective for any non-zero scalar $\alpha\in\mathbb C$. This concludes the claim.

Recall that the essential spectrum of an operator is stable under compact perturbation, so we have that $\sigma_{ess}(T)=\sigma_{ess}(T+G)=\{0\}$. It follows that $0\in\sigma(T+G)$ and is not an eigenvalue. Since $\sigma_{ess}(T+G)=\{0\}$, the spectrum of $T+G$ is at most countable, with $0$ the only possible accumulation point (see e.g. \cite{AA}, Corollary 7.50). In particular, $0\in\partial\sigma(T+G)$.  This shows that $T+G$ satisfies the conditions of Theorem~\ref{smallness-hilb}. Hence, there exists a rank one operator $F_1$, with $\norm{F_1}<\varepsilon/2$, such that $T+G-F_1$ has an invariant subspace of infinite dimension and codimension. Obviously, this implies the desired statement.
\end{proof}

\end{document}